\documentclass[12pt]{amsart}

\usepackage{microtype} %overfull boxes

\usepackage[left=3cm,right=3cm,bottom=2cm,top=3cm]{geometry}
\usepackage{amsmath}
\usepackage{amssymb}
\usepackage[mathscr]{euscript}
\usepackage[all]{xy}
\usepackage{graphicx}
\usepackage{latexsym}
\usepackage{graphics} %%%%%???
\usepackage{subfigure}
\usepackage{psfrag}  %%%%%
\usepackage{enumerate}
\usepackage{pstricks,pst-plot,pst-node}

\theoremstyle{plain}

\newtheorem{thm}{Theorem}[section]
\newtheorem{cor}[thm]{Corollary}
\newtheorem{lemma}[thm]{Lemma}
\newtheorem{prop}[thm]{Proposition}
\theoremstyle{definition}
\newtheorem{definition}[thm]{Definition}%[section]
\newtheorem{remark}[thm]{Remark}
\newtheorem{ex}[thm]{Example}%[section]
\theoremstyle{conjecture}
\def\B{\mathcal{B}}
\def\C{\mathcal{C}}
\def\D{\mathcal{D}}

\def\P{\mathcal{P}}

\def\sfC{\mathsf{C}}

\def\sfH{\mathsf{H}}

\def\sfT{\mathsf{T}}

\def\Hom{\operatorname{Hom}}

\def\Ext{\operatorname{Ext}}
\def\add{\operatorname{add}}
\def\ind{\operatorname{ind}}
\def\mod{\operatorname{mod}}

\begin{document}

\title[Hom-configurations and spherical objects]{Hom-configurations in triangulated categories generated by spherical objects}

\author{Raquel Coelho Sim\~oes}
\address{Centro de Estruturas Lineares e Combinat\'orias, Complexo Interdisciplinar
da Universidade de Lisboa, Av. Prof. Gama Pinto 2, 1649-003 Lisboa, Portugal.}
\email{rcoelhosimoes@campus.ul.pt}
\date{\today}

\begin{abstract}
Let $\sfT_w$ be a triangulated category generated by a $w$-spherical object, with $w \leq -1$. It is known that these categories are $w$-Calabi-Yau \cite{HJY}. Another instance of a triangulated category with negative Calabi-Yau dimension is the orbit category $\C_m (Q) := \D^b(Q) / \tau \Sigma^{m+1}$, where $Q$ is a (Dynkin) quiver, $\tau$ the AR-translate, $\Sigma$ the shift functor and $m \in \mathbb{N}$. We give a generalisation of the definition of Hom-configurations and Riedtmann configurations, which were studied in stable module categories of selfinjective algebras \cite{Riedtmann} and in $\C_1(Q)$ \cite{BRT, CS1}, and we classify these configurations in $\sfT_w$ in terms of arcs in the $\infty$-gon. We will also study a connection between these configurations  and the widely studied combinatorial objects known as noncrossing partitions. In addition, we prove that $\C_{|w|} (Q)$, where $Q$ is of type $A$, is a certain \textit{perpendicular} subcategory of $\sfT_w$, and use this fact to create a geometric model for these orbit categories. 
\end{abstract}



\maketitle

\bibliographystyle{plain}

\section{Introduction}

Calabi-Yau (CY, for short) triangulated categories appear in many branches of mathematics and physics, such as conformal field theory and string theory in theoretical physics, homological mirror symmetry in algebraic and symplectic geometry, and cluster-tilting theory in representation theory. Much work has been done on understanding triangulated categories of positive CY dimension, particularly those which are 2-CY or 3-CY. Thus far, little is understood about triangulated categories of negative CY dimension. In fact, there does not seem to exist a universally accepted definition of this concept yet. 

In this paper we will consider two classes of triangulated categories which appear to have negative CY dimension: triangulated categories $\sfT_w$ generated by a $w$-spherical object, and orbit categories $\C_{|w|} (Q) := \D^b(Q) / \tau \Sigma^{|w|+1}$, where $w$ is a negative integer, $Q$ is a Dynkin quiver, $\tau$ is the Auslander-Reiten translate and $\Sigma$ is the shift functor of the bounded derived category $\D^b(Q)$. 

The first class of categories, which are actually defined for any integer $w$, have recently been subject of intensive study (see \cite{FY}, \cite{HJ1}, \cite{HJY}, \cite{J}, \cite{KYZ}, \cite{Ng} and \cite{ZZ}), especially in connection with cluster-tilting theory, for $w \geq 2$. When $w \leq 0$, it is suggested in \cite{HJY} that these categories provide an example of categories with negative CY dimension, in the sense that $\Sigma^w$ is the \textit{only} power of the shift functor which is a Serre functor. The orbit category $\C_1(Q)$ was studied in \cite{CS1} (see also \cite{BRT}) and \cite{CS2}, also in connection with cluster-tilting theory and with the work of Riedtmann on the classification of selfinjective algebras.  

%Consider, for instance, $\C_1 (Q)$ where $Q$ is of type $A_n$. This can be viewed as the stable module category of a Nakayama algebra (for more details, see \cite{CS2}), and so by work of Dugas \cite[Theorem 6.1]{Dugas}, this category has positive CY dimension $2n-1$. 

%The main object of study in categories with positive CY-dimension is higher cluster-tilting objects, which satisfy certain positive Ext-vanishing properties. It would then be natural to study $(2n-1)$-cluster-tilting objects in $\C_1 (Q)$. However, the behaviour of these objects in this category is not very interesting. Indeed, given that, informally speaking, the size of the CY dimension is much larger than the \textit{size of the category}, $(2n-1)$-cluster-tilting objects are very small: they are indecomposable objects. Moreover, not every indecomposable object is a $(2n-1)$-cluster-tilting object. 

One reason that one should allow the CY dimension to be negative can be explained by the following. Consider the category $\C_1(A_n)$, which can be viewed as the stable module category of a Nakayama algebra (for more details, see \cite{CS2}). Taking the CY dimension to be the smallest \textit{positive} integer $d$ such that $\Sigma^d$ is the Serre functor, Dugas computes the CY dimension of $\C_1(A_n)$ to be $2n-1$  \cite[Theorem 6.1]{Dugas}.

However, taking this definition of CY dimension has several drawbacks. Most na\"ively, the CY dimension across the whole family, while given by a simple formula, is not uniform. This contrasts with the case for the cluster category $\D^b(Q) / \tau^{-1} \Sigma$, which is always $2$-CY, independently of the choice of quiver $Q$. A more serious drawback is the following: the natural objects of study in a $(2n-1)$-CY category are the $(2n-1)$-cluster-tilting objects. However, the properties of these objects in $\C_1(A_n)$ is not very interesting: they are indecomposable, and not every indecomposable object is a $(2n-1)$-cluster-tilting object. Moreover, since they are indecomposable objects, there is no natural theory of mutation. 

This drawback seems to be a consequence of, informally speaking, the CY dimension being much larger than the \textit{size} of the category. It turns out that both these shortcomings can be rectified by allowing the possibility of negative CY dimension. 

On the one hand, $\Sigma^w$ is a Serre functor of $\C_{|w|} (Q)$, independently of the choice of $Q$, and on the other hand, the behaviour of Hom-configurations and Riedtmann configurations in $\C_1 (Q)$, which are maximal/generating collections of indecomposable objects satisfying Hom-vanishing conditions for pairwise distinct objects, is highly reminiscent of that of cluster-tilting objects in the cluster category. For example, the number of pairwise non-isomorphic indecomposables in a Hom-configuration is precisely the number of vertices of $Q$, as is the case for cluster-tilting objects (which may be thought of as $\Ext^1$-configurations). For more details see \cite{BRT} and \cite{CS2}.

For this reason, we propose to study a more general version of Hom-configurations and Riedtmann configurations in $\sfT_w$, which will be called $|w|-\Hom$-configurations and $|w|$-Riedtmann configurations. We want to convey to the reader that these are the \textit{right} objects to study in the negative CY dimension world, by showing the nice combinatorics that these objects present in these categories. We will give a combinatorial classification of these two classes of objects, using the geometric model of $\sfT_w$ in terms of arcs in the $\infty$-gon \cite{HJ2}. In particular, we will see that, unlike in the orbit category $\C_1 (Q)$, where $Q$ is a Dynkin quiver, these two classes of objects do not coincide. 
%
%{\blue Simple-minded systems have been widely studied, mainly in the context of standard module categories of selfinjective algebras \cite{CKL}, but also in more general triangulated categories \cite{Dugas12}. In $\sfT_{-1}$, the set of simple-minded systems is a special class of $|-1|-$Riedtmann configurations, but we will see that these two notions do not coincide, giving a negative answer to a question of Dugas \cite{Dugas12}. We will see that simple-minded systems, when $w= -1$, have a very neat combinatorial classification in terms of arcs in the $\infty-$gon.}  

We will also prove that there is a strong connection between the two classes of categories studied in this paper, by showing that a certain \textit{perpendicular} subcategory of $\sfT_w$ is equivalent to the disjoint union of $\sfT_w$ itself and orbit category $\C_{|w|} (A_n)$. As an application of this fact, we will deduce a nice combinatorial model for the latter category, and consequently a new combinatorial description of the corresponding $|w|-$Hom and Riedtmann configurations, in terms of certain noncrossing pair partitions. We will compare this result with the already known bijection between Hom-configurations in $\C_1(A_n)$, and classical noncrossing partitions (see \cite{Riedtmann} and also \cite{CS2}). 

\section{Definitions and basic properties}

Let $w$ be a negative integer and $K$ be an algebraically closed field. Throughout this paper, $\sfT_w$ (or simply $\sfT$) will be a $K$-linear idempotent complete algebraic triangulated category which is classically generated by a $w$-spherical object, that is, an object $s$ satisfying
\[
\dim_K \, \Hom_{\sfT} (s, \Sigma^i s) = 
\begin{cases}
1 & \mbox{ if } i = 0, w \\
0 & \mbox{ \, } \text{otherwise.}
\end{cases}
\]
For more details see \cite[Section 0.a.]{HJY}. We will now state some properties of these categories. We refer the reader to the nice exposition in \cite[Section 1]{HJY} for more details. 

It is known that $\sfT$ is unique up to triangle equivalence \cite[Theorem 2.1]{KYZ}. Moreover, $\sfT$ is Krull-Schmidt and it has finite dimensional Hom-spaces. By \cite[Proposition 1.8]{HJY}, $\sfT$ is $w$-CY, and so we have the following bifunctorial isomorphism:
\[
\Hom_{\sfT} (x,y) \simeq D \Hom_{\sfT} (y, \Sigma^w x),
\]
for every object $x, y$ in $\sfT$. 

The AR-quiver of $\sfT$ consists of $|d|$ copies of $\mathbb{Z} A_{\infty}$, where $d:=w-1$. Given one copy, the others are obtained by applying $\Sigma, \ldots, \Sigma^{\mid d \mid -1}$ (cf. \cite[Proposition 1.10]{HJY}). 

The Hom-spaces in $\sfT$ are easy to compute. Namely, we have the following (cf. \cite[Proposition 2.2]{HJY}):
\begin{equation}
\dim_K \, \Hom_{\sfT} (x, y) = \begin{cases}
1 & \mbox{ if } y \in F^+ (x) \cup F^-(Sx) \\
0 & \mbox{ if } \text{otherwise.} \end{cases}
\end{equation}
Here $S$ denotes the Serre functor ($S = \Sigma^{w}$), $F^+$ denotes the forward Hom-hammock and $F^-$ the backward Hom-hammock (see \cite[Figure 3]{HJY}).\\

In \cite{HJ2}, the authors introduced a combinatorial model for $\sfT$, for $w \geq 2$, which can be easily tweaked for the case when $w \leq -1$. The indecomposable objects will be viewed as \textit{$d$-admissible arcs} of the $\infty$-gon, i.e. as pairs of integers $(t,u)$ with $t > u$ such that $t-u \geq \mid d \mid -1$ and $u - t \equiv 1 (\mod d)$. 

We shall identify indecomposable objects in $\sfT$, up to isomorphism, with the corresponding admissible arcs between vertices of the $\infty$-gon, and we shall freely switch between objects and arcs. We will use roman font for the indecomposable objects of $\sfT$ and typewriter font for the corresponding arcs.

Given an admissible arc $(t,u)$ corresponding to an indecomposable object of $\sfT$, we have: $\Sigma (t,u)  = (t-1,u-1)$, $\tau (t,u) = \Sigma^d (t,u) = (t-d,u-d)$ and $S(t,u) = \Sigma \tau (t,u) = \Sigma^w (t,u) = (t-w, u-w)$. The following figure shows one of the components of the AR-quiver of $\sfT$.

\begin{center}
\hspace{-1.75cm}$\xymatrix@!C=1em{
             & \vdots \ar[dr]         &                       & \vdots \ar[dr]       &                      & \vdots \ar[dr]
&                      & \vdots \ar[dr]       &                     & \vdots \ar[dr]     \\
 \ldots \ar[dr]\ar[ur]  &                       & *\txt{(-4d-1,0)}\ar[dr]\ar[ur] &                      & *\txt{(-3d-1,d)}\ar[dr]\ar[ur] &
& *\txt{(-2d-1,2d)}\ar[dr]\ar[ur] &                     & *\txt{(-d-1,3d)}\ar[dr]\ar[ur] &                          & \ldots \\
 & *\txt{(-4d-1,-d)} \ar[dr]\ar[ur] &                       & *\txt{(-3d-1,0)}\ar[dr]\ar[ur]  &                       & *\txt{(-2d-1,d)}\ar[dr]\ar[ur] &                      & *\txt{(-d-1,2d)}\ar[dr]\ar[ur] &                      & *\txt{(-1,3d)}\ar[dr]\ar[ur]  \\
\ldots \ar[dr]\ar[ur]     &  & *\txt{(-3d-1,-d)}\ar[dr]\ar[ur] &                       & *\txt{(-2d-1,0)}\ar[dr]\ar[ur]  &                      &  *\txt{(-d-1,d)}\ar[dr]\ar[ur]&
& *\txt{(-1,2d)}\ar[dr]\ar[ur]&                     & \ldots\\
 & *\txt{(-3d-1,-2d)}\ar[ur] &                       & *\txt{(-2d-1,-d)}\ar[ur] &                       & *\txt{(-d-1,0)}\ar[ur] &                      & *\txt{(-1,d)}\ar[ur] &                      & *\txt{(d-1,2d)}\ar[ur]}$
\end{center}

Given an indecomposable object $x$ of $\sfT$ and $i \in \mathbb{Z}$, we write: $x^{\perp_i} := \{ y \in \sfT \mid \Ext_{\sfT}^i (x,y) = 0\}$. Similarly, we write: $\,^{\perp_i} x := \{ y \in \sfT \mid \Ext_{\sfT}^i (y,x) = 0\}$.

We will now introduce the main object of study in this paper.

\begin{definition}
A collection $\sfH$ of indecomposable objects in $\sfT$ is called a \textit{$|w|-\Hom$-configuration} if it satisfies the following condition:
\[
h \in \sfH \Leftrightarrow \begin{cases} 
h \in (\sfH)^{\perp_i}, \text{ for } i = w+1, \ldots, -1 & \text{and } \\
h \in (\sfH \setminus \{h\})^{\perp_i} , \text{ for } i = 0, w.
\end{cases}
\]
\end{definition}

Since $\sfT$ is $w$-CY, the condition above is equivalent to
\[
h \in \sfH \Leftrightarrow \begin{cases} 
h \in \,^{\perp_i} (\sfH), \text{ for } i = w+1, \ldots, -1 & \text{and } \\
h \in \,^{\perp_i} (\sfH \setminus \{h\}), \text{ for } i = 0, w.
\end{cases}
\]
Note that when $w = -1$, we recover the definition of Hom-configuration in \cite{CS1}, i.e. a $|-1|-\Hom$-configuration is a maximal Hom-free collection of indecomposable objects.\\ 

\textbf{Notation.} Let $t$ and $u$ be two integers such that $(t,u)$ is a $d$-admissible arc. We call the collection of $d$-admissible arcs of the form $(x,u)$ with $x \geq t$ a \textit{partial right fountain at $u$ starting at $t$}, and we denote it by $\mathbb{RF}(u;t)$. Analogously, we call the collection of $d$-admissible arcs of the form $(t,y)$ with $y \leq u$ a \textit{partial left fountain at $t$ starting at $u$}, and we denote it by $\mathbb{LF}(t;u)$. \\

The following remark gives the list of the indecomposable objects/arcs lying in the forward and backward Hom-hammocks of a given indecomposable object in $\sfT$, in terms of these partial fountains.

\begin{remark}\label{Homhammocksfountains}
Let $x$ be an indecomposable object of $\sfT$ and $\mathtt{x} = (t,u)$ be the corresponding admissible arc. We have:
\begin{enumerate}
\item $F^+ (\mathtt{x}) = \bigcup \limits_{i = 0}^{\frac{u-t-1}{d}-1} \mathbb{LF}(t+id; u)$.
\item $F^- (\mathtt{x}) = \bigcup \limits_{i = 0}^{\frac{u-t-1}{d}-1} \mathbb{RF}(u-id; t)$.
\end{enumerate}
\end{remark}

The following lemma is an immediate consequence of this remark.

\begin{lemma}\label{cruciallemma}
Let $x, y$ be two indecomposable objects of $\sfT$, $\mathtt{x} = (t,u), \mathtt{y}$ be the corresponding arcs, and $j \in \mathbb{Z}$. Then we have $\Ext^j (x, y) \ne 0$ if and only if $\mathtt{y} \in \mathbb{LF}(V_{j};u+j) \cup \mathbb{RF}(V_{j};t-d-1+j)$, where $V_{j} := \{t+id+j \mid \, i= 0, \ldots, \frac{u-t-1}{d}-1 \}$. 
\end{lemma}

\begin{remark}
It follows from lemma \ref{cruciallemma} that when $w \leq -2$ and $x \in \ind \, \sfT$, we have $\Ext^i (x,x) = 0$ for all $i \in \{w+1, \ldots, -1\}$. %Hence, given any indecomposable in $\sfT$ there is always a $|w|-\Hom$-configuration containing it. 
\end{remark}

\section{Classification of $|w|-\Hom$-configurations in $\sfT$}

In this section we will give a description of $|w|-\Hom$-configurations in $\sfT$ in terms of admissible arcs in the $\infty$-gon. %The proof will be divided into three lemmas. 

\begin{lemma}\label{combinatorialExtifree}
Let $x, y$ be two indecomposable objects in $\sfT$. Then $y \in \bigcap \limits_{i= w}^0 x^{\perp_i}$ if and only if the following conditions hold:
\begin{enumerate}
\item $\mathtt{x}$ and $\mathtt{y}$ do not cross,
\item $\mathtt{x}$ and $\mathtt{y}$ are not incident with the same vertex (of the $\infty$-gon). 
\end{enumerate}
\end{lemma}
\begin{proof}
The proof follows immediately from lemma \ref{cruciallemma}. Indeed, note that $\bigcup \limits_{j=w}^{0} V_j = [u,t] \cap \mathbb{Z}$, and all arcs satisfying the properties in lemma \ref{cruciallemma} either cross $\mathtt{x} = (t,u)$ or are incident with $t$ or $u$. 
\end{proof}

\begin{lemma}\label{isolatedvertexoverarc}
Let $\sfH$ be a $|w|-\Hom$-configuration in $\sfT$. Given an indecomposable object $x$ in $\sfH$, there are precisely $|w|-1$ isolated vertices whose smallest overarc is $\mathtt{x}$. 
\end{lemma}
\begin{proof}
Suppose $\mathtt{x} = (t,u)$ has minimum length $|d|-1 = |w|$. This means that there are precisely $|w|-1$ vertices in $]u,t[$, and these vertices are necessarily isolated. 

Now let $t-u = k |d| - 1$, for some $k \geq 2$ and suppose, for a contradiction, that there are at least $|w|$ isolated vertices in $]u,t[$ whose smallest overarc is $(t,u)$. Let $v_1, v_2, \ldots, v_{|w|}$, with $v_i < v_{i+1}$, be $|w|$ isolated vertices satisfying the property above and such that there are no isolated vertices in each interval $]v_i, v_{i+1}[$ whose smallest overarc is $(t,u)$.  

Let $i \in \{1, \ldots, |w|-1\}$. Since $v_i$ and $v_{i+1}$ have no overarcs inside $]u,t[$, there must be $k_i |d|$ vertices in the interval $]v_i,v_{i+1}[$, for some $k_i \geq 0$. 

Suppose there are no more isolated vertices in $]u,t[$ whose smallest overarc is $(t,u)$. Then there are $k_0 |d|$ and $k_{|w|} |d|$ vertices in $]u,v_1[$ and $]v_{|w|}, t[$, respectively, for some $k_0, k_{|w|} \geq 0$. Therefore, the number of vertices in $[u,t]$ is given by $|w| + 2 + k^\prime |d| = |d| + 1 + k^\prime |d|$, where $k^\prime = \sum_{i= 0}^{|w|} k_i$. On the other hand, the number of vertices in $[u,t]$ must be of the form $l |d|$, for some $l > 0$, since $(t,u)$ is an admissible arc. Hence $|d| (l - k^\prime -1) = 1$, a contradiction as $|d| \geq 2$. Hence, there must be at least one more isolated vertex $v_0$, which we can assume, without loss of generality, to be strictly smaller than $v_1$, and such that there are no isolated vertices in $]v_0,v_1[$ whose smallest overarc is $(t,u)$.

We claim that $(v_{|w|}, v_0)$ is an admissible arc. Indeed, by the same argument as above, there are $k_0 |d|$ vertices in $]v_0,v_1[$, for some $k_0$, and so the number of vertices in $[v_0,v_{|w|}]$ is given by $|d| (1 + \sum_{i = 0}^{|w|-1} k_i)$. 

Hence $(v_{|w|}, v_0)$ is an admissible arc, it does not cross any arc in $\sfH$, nor does it have a common vertex with another arc in $\sfH$, as $v_0$ and $v_{|w|}$ are both isolated. We have reached a contradiction, since $(v_{|w|},v_0) \not\in \sfH$ and $\sfH$ is a $|w|-\Hom$-configuration. This finishes the proof that there are at most $|w|-1$ isolated vertices whose smallest overarc is $(t,u)$. 

Now suppose there are only $|w|-i$ isolated vertices whose smallest overarc is $(t,u)$, for some $ 2 \leq i \leq |w|$. Using the same argument as above, one concludes that there must be $i + 2 + k |d|$ vertices in $[u,t]$, for some $k \geq 0$. However, $(t,u)$ is admissible and so there are $k^\prime |d|$ vertices in $[u,t]$, for some $k^\prime > k$. Hence, we have $(k^\prime - k) = i + 2 \leq |d|-1$, a contradiction as $k^\prime > k$.
\end{proof}

\begin{lemma}\label{isolatedverticesHomconfigs}
A $|w|- \Hom$-configuration in $\sfT$ has at most $|w|$ isolated vertices with no overarc. 
\end{lemma}
\begin{proof}
Suppose, for a contradiction, that there are $|w|+1$ isolated vertices with no overarc. Let $v_1, \ldots, v_{|w|+1}$ be vertices satisfying such conditions and suppose that $v_i < v_{i+1}$ and that there are no isolated vertices in $]v_i, v_{i+1}[$ with no overarc. Then, the number of vertices in each interval $]v_i,v_{i+1}[$ is of the form $k_i |d|$, for some $k_i \geq 0$, and so the number of vertices in $[v_1, v_{|w|+1}]$ is given by $|w| + 1 + k |d| = (k+1) |d|$, where $k = \sum_{i=1}^{|w|} k_i$. This shows that the arc $(v_{|w|+1}, v_1)$ is admissible. Moreover, this arc neither crosses nor does it have a common vertex with any other arc of the $|w|-\Hom$-configuration $\sfH$. Therefore $\sfH$ is not a $|w|-\Hom$-configuration, a contradiction.
\end{proof}

The characterisation of $|w|-\Hom$-configurations in $\sfT$ in terms of arcs of the $\infty$-gon is given in the following theorem.

\begin{thm}\label{Homconfigclassification}
A set of admissible arcs $\sfH$ is a $|w|-\Hom$-configuration in $\sfT$ if and only the following conditions hold:
\begin{enumerate}
\item There are no crossings nor arcs incident with the same vertex,
\item Given an arc $\mathtt{a}$ corresponding to an indecomposable object in $\sfH$, there are precisely $|w|-1$ isolated vertices whose smallest overarc is $\mathtt{a}$,
\item There are at most $|w|$ isolated vertices with no overarc. 
\end{enumerate}
\end{thm}
\begin{proof}
Lemmas \ref{combinatorialExtifree}, \ref{isolatedvertexoverarc} and \ref{isolatedverticesHomconfigs} give us the necessary conditions. Now, let $\sfH$ be a collection of indecomposable objects of $\sfT$ whose corresponding arcs satisfy conditions $(1), (2)$ and $(3)$, and suppose for a contradiction that $\sfH$ is not a $|w|-\Hom$-configuration. By lemma \ref{combinatorialExtifree} and condition $(1)$, this means that there is an indecomposable object $x$ in $\sfT \setminus \sfH$ such that $x \in (\sfH)^{\perp_i}$, for all $i = w, \ldots, 0$. Let $\mathtt{x} = (t,u)$ and let $\sfH^\prime$ be a $|w|-\Hom$-configuration containing $\sfH \cup \{x\}$. Note that $t$ and $u$ must be isolated vertices in $\sfH$, and they either share the smallest overarc $\mathtt{a}$ or they do not have any overarc in $\sfH$. However, the first case contradicts lemma \ref{isolatedvertexoverarc} for $\sfH^\prime$, since there are fewer than $|w|-1$ isolated vertices whose smallest overarc is $\mathtt{a}$ by condition $(2)$. Hence, $t$ and $u$ are isolated vertices in $\sfH$ with no overarcs. Since $\sfH^\prime$ is a $|w|-\Hom$-configuration, by lemma \ref{isolatedvertexoverarc} there are precisely $|w|-1$ isolated vertices in $]u,t[$. Note that these vertices are isolated with no overarcs in $\sfH$ since $(t,u)$ is their smallest overarc in $\sfH \cup \{\mathtt{x}\}$. Therefore, $\sfH$ has $|w|+1$ isolated vertices with no overarcs, which contradicts condition $(3)$.
\end{proof}

\begin{ex}\label{canonicalexamples}
Let $w = -1$. Then the admissible arcs are the arcs with odd length. There are two canonical classes of examples of $|-1|-\Hom$-configurations in $\sfT$. The first one is of the form $\sfH_1 := \{ (j, j-1) \mid j = 2k, k \in \mathbb{Z}\}$ (or of the form $\{ (j,j-1) \mid j = 2k-1, k \in \mathbb{Z} \}$), which corresponds to every object in the mouth of one of the two components of the AR-quiver of $\sfT$. The second one is of the form $\sfH_2:= \{ (j+1,j) \mid j = i + (2k+1), k \in \mathbb{N}_0\} \cup \{(j, j-1) \mid j = i + (2k+1), k \in \mathbb{Z}_{<0} \}$, for a fixed integer $i$. The first subset of $\sfH_2$ corresponds to arcs in the mouth of one of the components of the AR-quiver of $\sfT$, whilst the second subset corresponds to arcs lying in the mouth of the other component.

It is natural to ask whether there is a representation-theoretic difference between these two classes of examples. The answer will be given in the next section.
\end{ex}

\section{Riedtmann configurations}

It was seen in \cite{BRT} and \cite{CS1} that the notion of Hom-configurations in the orbit category $\C_1 (Q)$, where $Q$ is of Dynkin type, coincides with that of Riedtmann configurations, which are certain collections of objects introduced in \cite{Riedtmann} in order to classify selfinjective algebras of finite representation type. 

In this section we will give a generalization of this concept and compare it with $|w|-\Hom$-configurations. Unlike the case in the orbit category, we do not have a one-to-one correspondence between these two classes of objects, but we will see in which circumstances a $|w|-\Hom$-configuration is a $|w|-$Riedtmann configuration. 

\begin{definition}\label{def:Riedtmannconf}
A collection $\sfC$ of indecomposable objects of $\sfT$ is called a:
\begin{enumerate}
\item \textit{left $|w|-$Riedtmann configuration} if:
\begin{enumerate}[(a)]
\item $\Ext^i_{\sfT} (x,y) = 0$, for all $x, y \in \sfC, x \ne y, i \in \{w, \ldots, 0\}$.
\item For every $z \in \sfT$, there is an $x \in \sfC$ and $i \in \{w+1, \ldots, 0\}$ such that $\Ext^i_{\sfT} (x, z) \ne 0$. 
\end{enumerate}
\item \textit{right $|w|-$Riedtmann configuration} if it satisfies (a) and:
\begin{enumerate}[(b$^\prime$)]
\item For every $z \in \sfT$, there is an $x \in \sfC$ and $i \in \{w+1, \ldots, 0\}$ such that $\Ext^i_{\sfT} (z,x) \ne 0$. 
\end{enumerate}
\end{enumerate}
\end{definition}

It is obvious that any left or right $|w|-$Riedtmann configuration is a $|w|-\Hom$-configuration.

When $w = -1$, we recover Riedtmann's original definition \cite{Riedtmann}. Note that, if we let $i$, in (b) and (b$^\prime$), range from $\left \lfloor \frac{w}{2} \right \rfloor$ to $0$ instead, we would get another generalisation of Riedtmann's original definition. Let us call this different notion by \textit{alternative} left/right $|w|-$Riedtmann-configuration. The reason why we chose definition \ref{def:Riedtmannconf} to be the generalisation of Riedtmann's configuration is explained in the following lemma. 

\begin{lemma}
There are no alternative left/right $|w|-$Riedtmann-configurations in $\sfT$.
\end{lemma}
\begin{proof}
We only prove that $\sfT$ has no alternative right $|w|-$Riedtmann configurations, as the argument for alternative left$|w|-$Riedtmann configurations is similar. Suppose, for a contradiction, that there is a alternative right $|w|-$Riedtmann configuration $\sfH$. Then, in particular, $\sfH$ is a $|w|-\Hom$-configuration, and so there must be an arc $(t,u)$ in $\sfH$ of minimum length $|w|$. Consider $z \in \sfT$ whose corresponding arc is $(t-1,u-1)$. We claim that there is no $x \in \sfH$ such that $\Ext_{\sfT}^i (z,x) \ne 0$, for some $i \in \{ \left \lfloor \frac{w}{2} \right \rfloor, \ldots, 0\}$. Indeed, $x$ satisfies this condition if and only if $\mathtt{x} \in \mathbb{LF}(t+i-1,u+i-1) \cup \mathbb{RF}(t+i-1;t-1-w+i)$. But any arc satisfying this condition, for any $i \in \{\left \lfloor \frac{w}{2} \right \rfloor, \ldots, 0\}$, either crosses $(t,u)$ or is incident with $t$ or $u$, and so it cannot lie in $\sfH$. 
\end{proof} 

\begin{thm}
Let $\sfH$ be a $|w|-\Hom$-configuration in $\sfT$. The following are equivalent:
\begin{enumerate}
\item $\sfH$ is a left $|w|-$Riedtmann configuration,
\item $\sfH$ is a right $|w|-$Riedtmann configuration,
\item There are at most $|w|-1$ isolated vertices with no overarcs. 
\end{enumerate}
\end{thm}
\begin{proof}
$(2) \Rightarrow (3)$: Suppose $\sfH$ has precisely $|w|$ isolated vertices with no overarc. Let $v_1, \ldots, v_{|w|}$, with $v_i < v_{i+1}$ be such vertices and consider the arc $\mathtt{a} = (v_{|w|}, v_1-1)$. Since there are $k_i |d|$ vertices in each interval $]v_i,v_{i+1}[$, for some $k_i \geq 0$, the number of vertices in $[v_1-1,v_{|w|}]$ is given by $|d| (k+1)$, where $k = \sum_{i=1}^{|w|-1} k_i$. Hence the arc $\mathtt{a}$ is admissible. 

We claim that $\Ext^i_{\sfT} (a, \sfH) = 0$, for all $i = w+1, \ldots, 0$. Indeed, given $h \in \sfH$ and $i \in \{w+1, \ldots, 0\}$, we have $\Ext^i_{\sfT} (a,h) \ne 0$ if and only if $\mathtt{h} \in \mathbb{LF}(V_i; v_1-1+i) \cup \mathbb{RF}(V_i; v_{|w|}-w+i)$, where $V_i = \{v_1-1-w+i, \ldots, v_{|w|}+d+i, v_{|w|}+i\}$. There cannot be any indecomposable $h \in \sfH$ satisfying these conditions, since the isolated vertices $v_1, \ldots, v_{|w|}$ have no overarcs. Therefore $\sfH$ is not a right $|w|-$Riedtmann configuration. 

$(1) \Rightarrow (3)$: As above, using the arc $(v_{|w|}+1, v_1)$ instead of $\mathtt{a}$.

$(3) \Rightarrow (2)$: Let $\mathtt{a} = (t,u) \in \sfT$, with $t-u = k |d| -1$, for some $k \geq 1$. We would like to prove that there is $x \in \sfH$ and $i \in \{w+1, \ldots, 0\}$ such that $\Ext^i_{\sfT} (a, x) \ne 0$. 

By lemma \ref{cruciallemma}, we have that, for each $i \in \{w, \ldots, 0\}$, $\Ext^i (a, x) \ne 0$ if and only if $\mathtt{x} \in \mathbb{LF}(V_i;u+i) \cup \mathbb{RF}(V_i;t+i-w)$, with $V_i = \{t+i+jd \mid j= 0, \ldots, k-1\}$. 

Let $z_j := t+w+(j-1)d$, with $j = 1, \ldots, k$. Note that these vertices are the elements of $V_w$ and that $z_k = u$. Denote by $X_1$ the set of integers in the interval $]z_1,t]$ and by $X_j$ the set of integers in $]z_i,z_{i-1}[$, for $j = 2, \ldots, k$. See figure \ref{exXis} for an example in $\sfT_{-2}$, where $(t,u) = (11,0)$.

\begin{figure}[!ht]
\psfragscanon
\psfrag{X1}{\small{$X_1$}}
\psfrag{X2}{\small{$X_2$}}
\psfrag{X3}{\small{$X_3$}}
\psfrag{X4}{\small{$X_4$}}
\includegraphics[scale=.8]{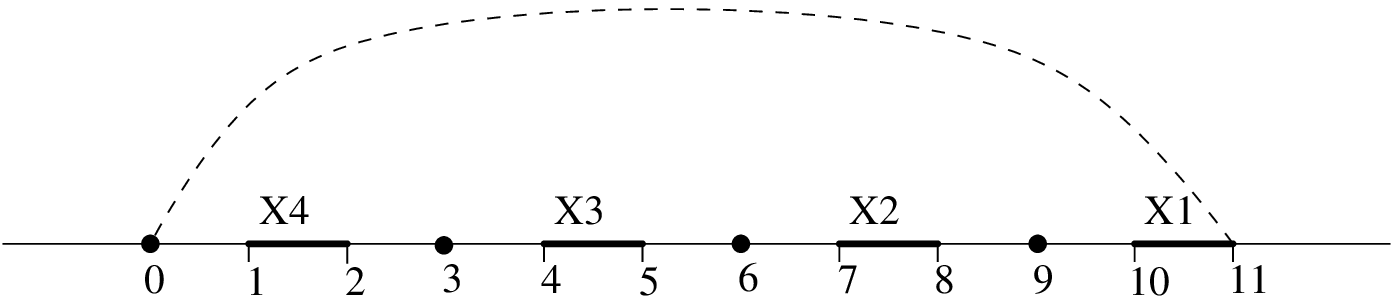}
\caption{An example in $\sfT_{-2}$.}
\label{exXis}
\end{figure}

With this new notation, we say that $\Ext^i (a,x) \ne 0$, for some $i \in \{w+1, \ldots, 0\}$, if and only if:
\begin{enumerate}[(I)]
\item $\mathtt{x}$ is incident with a vertex in $\bigcup \limits_{j=1}^k X_j$,
\end{enumerate}
and one of the following conditions hold:
\begin{enumerate}
\item[(II)] $\mathtt{x}$ crosses $\mathtt{a}$ or is incident with $u$,
\item[(III)] $\mathtt{x}$ is of the form $\mathtt{x} = (t,v)$ with $v<u$ or $\mathtt{x} = (v,t)$ with $v>t$. 
\end{enumerate}

First, let us assume that $\mathtt{a}$ has minimum length, i.e. $k = 1$. On one hand, since $\sfH$ is a $|w|-\Hom$-configuration, there are precisely $|w|-1$ isolated vertices sharing the same smallest overarc. On the other hand, by hypothesis, there is at most $|w|-1$ isolated vertices with no overarc. Therefore, since $X_1$ has $|w|$ consecutive integers, there must be an arc $\mathtt{x}$ in $\sfH$ incident with one of the vertices in $X_1$. Since any arc incident with one vertex in $X_1$ satisfies conditions $(I)$ and either $(II)$ or $(III)$, we are done. 

Now let $k >1$, and suppose, for a contradiction, that every arc in $\sfH$ fails one of the conditions above. Since $X_1$ has $|w|$ consecutive integers, one of them must be non-isolated. Let $v_1$ be the largest non-isolated vertex in $X_1$. 

\textbf{Case 1.} $v_1 \ne t$. 

Let $\mathtt{a}_1$ be the arc in $\sfH$ incident with $v_1$. By assumption, $\mathtt{a}_1$ does not satisfy $(II)$, and so $\mathtt{a}_1$ must be of the form $(v_1,v_{i_1})$, with $v_{i_1} \in X_{i_1}$, $i_1 \in \{2, \ldots, k\}$. If every integer in $]z_{i_1}, v_{i_1}[$ is isolated, then the number of isolated vertices in $]z_{i_1}, t]$ which do not have $\mathtt{a}_1$ as an overarc is given by $(i_1 |d|-1) - ((i_1-1)|d|) = |w|$, where the first summand is the number of integers in $]z_{i_1},t]$ and the second summand is the number of integers in $[v_{i_1},v_1]$. Note that these $|w|$ isolated vertices either do not have any overarc, or they share an overarc, which contradicts the hypothesis and the fact that $\sfH$ is a $|w|-\Hom$-configuration. Therefore, there is at least one non-isolated vertex in $]z_{i_1},v_{i_1}[$. Let $v^\prime_{i_1}$ be the largest one and $\mathtt{a}_2$ the arc incident with it. If $i_1 = k$, then any arc incident with $v_{i_1}^\prime$ satisfies $(I)$ and $(II)$, a contradiction. If $i_1 < k$, then, by assumption, $\mathtt{a}_2$ is of the form $(v_{i_1}^\prime, v_{i_2})$ with $v_{i_2} \in X_{i_2}$, $i_1 < i_2 \leq k$. Again, if the integers in $]z_{i_2},v_{i_2}[$ are all isolated, then we would have $|w|$ isolated vertices in $]z_{i_2}, t]$ with either a common overarc or with no overarc, which is a contradiction. Hence there is a non-isolated vertex in $]z_{i_2},v_{i_2}[$.

Proceeding this way, we get a sequence of arcs $\mathtt{a_1}, \ldots, \mathtt{a}_r$, such that $\mathtt{a}_r = (v^\prime_{i_{r-1}},v_k)$, with $v_k \in X_k$. Using the same argument as above, one of the integers in $]u,v_k[$ must be non-isolated, but any arc incident with that integer satisfies $(I)$ and $(II)$, a contradiction. We are then left with the following case.

\textbf{Case 2.} $v_1 = t$. 

By assumption, the arc $\mathtt{a}_1$ incident with $v_1$ does not satisfy $(III)$, and so it must be of the form $\mathtt{a}_1 = (t,z_{i_1})$, for some $i_1 \in \{1, \ldots, k-1\}$, because $\mathtt{a}_1$ is admissible. Since $X_{i_1+1}$ has $|w|$ (consecutive) elements, one of them must be non-isolated. Let $v_{i_1+1}$ be the largest one, and let $\mathtt{a}_2 = (v_{i_1+1},v_{i_2})$ be the arc incident with it, where $v_{i_2} \in ]u,z_{i_1+1}]$.

\textbf{Subcase 2.1.} $v_{i_2} \not\in V_w$. Then we fall into the same argument as the one used in case 1. 

\textbf{Subcase 2.2.} $v_{i_2} = z_{i_2}$, for some $i_2 \in \{i_1+1, \ldots, k\}$. 

Since $X_{i_2+1}$ has $|w|$ (consecutive) elements, at least one of them must non-isolated. Let $v_{i_2+1}$ be the largest one, and $\mathtt{a}_3$ be the arc incident with it. Proceeding in this manner, we get a sequence of arcs $\mathtt{a}_1, \ldots, \mathtt{a}_r$, with $\mathtt{a}_j = (v_{i_{j-1}+1}, z_{i_j})$, for $j >1$, and $z_{i_r} = z_{k-1}$. Now, $X_k$ has $|w|$ elements, and therefore one of them must be non-isolated. But any arc incident with these elements satisfy $(I)$ and $(II)$, contradiction.

The argument to prove $(3) \Rightarrow (1)$ is analogous.
\end{proof}

\section{The double-perpendicular category}

In this section, we will show that there is a strong relationship between the categories $\sfT$ and $\C_{|w|} (Q)$, where $Q$ is a Dynkin quiver of type $A$.

Let $x \in \ind \sfT$ and denote by $x^{\perp}$ the perpendicular category of $x$ defined by
\[
x^\perp := \bigcap \limits_{i=w}^0 x^{\perp_i}.
\]

It is easy to check that the objects in the subcategory $\ind (\mod (KQ) \cup \Sigma (\mod (KQ)) \cup \ldots \cup \Sigma^{|w|} (\mod (KQ) \setminus \mathcal{I}))$ of $\D^b(Q)$, where $\mathcal{I}$ denotes the set of injective modules, is a fundamental domain for the action of $\tau \Sigma^{|w|+1}$ on $\ind \D^b(Q)$. 

In the proof of the following theorem, we will use the fact that every indecomposable module over $KQ$ is uniquely determined by its simple socle and its composition length, since $KQ$ is a Nakayama algebra. The indecomposable $KQ$-module with length $l$ and socle $S_{a_1}$, which is the simple module at vertex $a_1$ of $Q$, will be denoted by the sequence $(a_l, \ldots, a_2, a_1)$, where $a_i = a_{i-1}+1$, for $i = 2, \ldots, l$.   

\begin{thm}\label{thm:doubleperp}
Let $\mathtt{a} = (t,u)$ be the arc corresponding to $a \in \ind \, \sfT$. Then the arcs corresponding to the indecomposable objects in $a^\perp$ are $\sfC_1 \cup \sfC_2$, where:
\[
\sfC_1 := \{ \text{d-admissible arcs } (x,y) \mid u < y < x < t \},
\]
\[
\sfC_2:= \{ \text{d-admissible arcs } (x,y) \mid u < y < x < t \}.
\]
Let $C_i$ be the set of indecomposable objects of $\sfT$ corresponding to $\sfC_i$, where $i = 1, 2$. Then $\add \, C_2$ is equivalent to $\sfT$ and $\add \, C_1$ is equivalent to $\C_{|w|} (Q)$, where $Q$ is the quiver $n \rightarrow n-1 \rightarrow \cdots \rightarrow 1$ of type $A_n$ with $n = \frac{u-t-1}{d}-1$. 
\end{thm}
\begin{proof}
Note that the indecomposable objects of $x^\perp$ correspond to all the admissible arcs which do not cross nor have a common vertex with $\mathtt{x}$, so the first statement of this theorem is clear. 

Before defining a $K$-linear functor $\mathcal{F}$ from $\C_{|w|} (Q)$ to $\add C_1$, we note that an indecomposable object $X$ lying in the fundamental domain of $\C_{|w|} (Q)$ must be of the form $X = \Sigma^{d(X)} \overline{X}$, where $\overline{X} = (a_l, \ldots, a_2,a_1)$, $d(X)$ denotes the degree of $X$, which lies in $\{0, \ldots, |w|\}$, and $a_l \ne n$ when $i = |w|$.

Given an indecomposable object $X = \Sigma^i (a_l, \ldots, a_2, a_1)$ in $\C_{|w|}$, we define $\mathcal{F}(X)$ to be the indecomposable object corresponding to the arc $(t-i-1+(a_1-1)d,u-i-1-(n+2-l-a_1)d)$, which we denote by $\mathcal{F}(\mathtt{X})$. One can easily check that these arcs lie in $\sfC_1$ and that $\mathcal{F}$ is surjective on the objects. 

It is enough to define $\mathcal{F}$ on the morphisms between indecomposable objects. In order to define the functor on morphisms, it is enough to check that given $M, N \in \ind \, \C_{|w|} (Q)$, $\Hom_{\C_{|w|} (Q)} (M,N) \ne 0$ if and only if $\Hom_{\sfC_1} (\mathcal{F}(M),\mathcal{F}(N)) \ne 0$, since the dimensions of the homomorphism spaces in both categories are either zero or one-dimensional. Let $\overline{M} = (a_l, \ldots, a_2, a_1)$ and $\overline{N} = (b_m, \ldots, b_2, b_1)$. We have $\Hom_{\C_{|w|} (Q)} (M,N) \ne 0$ precisely when $M$ and $N$ satisfy one of the following three conditions:

\begin{enumerate}
\item $d(M) = d(N) = i$ and there is some $1 \leq j \leq l$ such that $(a_l, \ldots, a_j) = (b_{l-j+1}, \ldots, b_1)$.
\item $d(M) = i, d(N) = i+1$, $0 \leq i \leq |w|-1$, $b_1 \leq a_1-1$ and $a_1-1 \leq b_m \leq a_l-1$. 
\item $d(M) = |w|, d(N) = 0$, $b_1 \leq a_1$ and $a_1 \leq b_m \leq a_l$.  
\end{enumerate}

\textbf{Case 1.} $d(M) = d(N) = i$, for some $0 \leq i \leq |w|$. 

Then we claim that the following are equivalent:
\begin{enumerate}[(a)]
\item There is some $1 \leq j \leq l$ such that $(a_l, \ldots, a_j) = (b_{l-j+1}, \ldots, b_1)$.
\item $\mathcal{F}(N) \in F^+ (\mathcal{F}(M))$.
\item $\Hom_{\sfC_1} (\mathcal{F}(M),\mathcal{F}(N)) \ne 0$.
\end{enumerate}

Firstly we prove that (a) $\Longleftrightarrow$  (b). We have $\mathcal{F}(\mathtt{M}) = (t-i-1+(a_1-1)d,u-i-1-(n+2-l-a_1)d) =: (x,y)$ and $\mathcal{F}(\mathtt{N}) = (t-i-1 + (b_1-1) d, u-i-1- (n+2-m-b_1)d) =: (x^\prime, y^\prime)$. Note that, by remark \ref{Homhammocksfountains}, condition (b) holds if and only if $x, y, x^\prime$ and $y^\prime$ satisfy the following:

\begin{enumerate}[(i)]
\item $x^\prime - x = j^\prime d$, for some $j^\prime \geq 0$. 
\item $x^\prime \geq y - w$.
\item $y^\prime \leq y$. 
\end{enumerate}

We have $x^\prime - x = (b_1-a_1) d$, and $j^\prime := b_1 - a_1  \geq 0$ if and only if $b_1 \geq a_1$. On the other hand, since $n = \frac{u-t-1}{d}-1$, we have $x^\prime - y +w = (b_1-a_l) d \geq 0$ if and only if $b_1 \leq a_l$. Hence (i) and (ii) is equivalent to having $a_1 \leq b_1 \leq a_m$. Finally, we have $y^\prime-y = (m-l+a_j-a_1) d = (b_m-a_l) d$, since $l = a_l-a_1+1$ and $m = b_m-b_1+1$. So, $y^\prime-y \leq 0$ if and only if $b_m \geq a_l$. We can then conclude that conditions (i), (ii) and (iii) are equivalent to (a), which proves that (a) $\Longleftrightarrow$ (b). 

Now, suppose (c) holds. Then $\mathcal{F}(N) \in F^+ (\mathcal{F}(M)) \cup F^-(S(\mathcal{F}(M)))$. Assume, for a contradiction, that $\mathcal{F}(N) \in F^-(S(\mathcal{F}(M)))$. Note that $S(\mathcal{F}(M))$ is the indecomposable object corresponding to the arc $(t-i-2+(a_1-2)d,u-i-2-(n+3-l-a_1)d) =: (x^{\prime\prime},y^{\prime\prime})$. Then, in particular, we must have $y^\prime-y^{\prime\prime} = j^\prime d$, for some $j^\prime \leq 0$, by remark \ref{Homhammocksfountains}. However, $y^\prime -y^{\prime \prime} = 1 + (b_m-a_l+1)d$ which cannot be written in the form $j^\prime d$ for some $j^\prime \leq 0$, since $d \leq -2$. This shows (c) $\implies$ (b) and the converse is trivial. Therefore, in case 1, we have that $\Hom_{\C_{|w|} (Q)} (M,N) \ne 0$ is equivalent to $\Hom_{\sfC_1} (\mathcal{F}(M),\mathcal{F}(N)) \ne 0$.

\textbf{Case 2.} $d(M) = i, d(N) = i+1$, for some $0 \leq i \leq |w|-1$.

Using similar arguments to the ones used in case 1, one can show that the following are equivalent:
\begin{enumerate}[(a)]
\item $b_1 \leq a_1-1$ and $a_1-1 \leq b_m \leq a_l-1$.
\item $\mathcal{F}(N) \in F^- (S(\mathcal{F}(M)))$.
\item $\Hom_{\sfC_1} (\mathcal{F}(M),\mathcal{F}(N)) \ne 0$.
\end{enumerate}

As we have mentioned above, (a) is equivalent to $\Hom_{\C_{|w|} (Q)} (M,N) \ne 0$, and so we are done in this case. The remaining case is the following:

\textbf{Case 3.}  $d(M) = |w|$, $\Sigma^w M $ is not an injective module and $d(N) = 0$.

Then one can easily check that the following are equivalent:
\begin{enumerate}[(a)]
\item $b_1 \leq a_1$ and $a_1 \leq b_m \leq a_l$.
\item $\mathcal{F}(N) \in F^- (S(\mathcal{F}(M)))$.
\item $\Hom_{\sfC_1} (\mathcal{F}(M),\mathcal{F}(N)) \ne 0$.
\end{enumerate}

Therefore, we have that $\Hom_{\C_{|w|} (Q)} (M,N) \ne 0$ is equivalent to $\Hom_{\sfC_1} (\mathcal{F}(M),\mathcal{F}(N)) \ne 0$. 

It is easy to check that $\mathcal{F}$ is indeed a functor, which is full and faithful, which thus gives an equivalence between $\add \sfC_1$ and $\C_{|w|} (Q)$, as desired. 

%{\red More details on the arguments omitted: see notes 4/7/13.}

An equivalence between $\add \sfC_2$ and $\sfT$ can be clearly induced from the bijection between $\mathbb{Z} \cap (] - \infty, u-1[ \, \cup \, ]t+1, + \infty[)$ and $\mathbb{Z}$ described as follows: an object in $\mathbb{Z} \cap (] - \infty, u-1[ \, \cup \, ]t+1, + \infty[)$ is either of the form $t+i$ or of the form $u-i$, for some $i \geq 1$. The bijection is then given by: $t+i \mapsto i-1$, $u-i \mapsto -i$. One can easily do a case-by-case analysis to check that the induced bijection on the objects behaves well with morphisms.
\end{proof}

\begin{ex}
Let $w = -1$, $(t,u) = (3,-4)$ and $Q = 3 \rightarrow 2 \rightarrow 1$. Then the AR-quiver of $\C_1(Q)$ is as follows:
\[
\xymatrix@C=0.2cm{& & (1,2,3) \ar@{->}[dr] & & \Sigma (1) \ar@{->}[dr] & & \Sigma (2) \ar@{->}[dr] & & (1,2,3) \\
& (1,2) \ar@{->}[dr]\ar@{->}[ur] & & (2,3) \ar@{->}[dr]\ar@{->}[ur] & & \Sigma (1,2) \ar@{->}[dr]\ar@{->}[ur] & & (1,2) \ar@{->}[ur] \\
(1) \ar@{->}[ur] & & (2) \ar@{->}[ur] & & (3) \ar@{->}[ur] & & (1) \ar@{->}[ur]}
\]
The functor $F$ maps the AR-quiver above to the AR-quiver of $\add C_1$ which is given by:
\[
\xymatrix@C=0.1cm{& & (2,-3) \ar@{->}[dr] & & (1,0) \ar@{->}[dr] & & (-1,-2) \ar@{->}[dr] & & (2,-3) \\
& (2,-1) \ar@{->}[dr]\ar@{->}[ur] & & (0,-3) \ar@{->}[dr]\ar@{->}[ur] & & (1,-2) \ar@{->}[dr]\ar@{->}[ur] & & (2,-1) \ar@{->}[ur] \\
(2,1) \ar@{->}[ur] & & (0,-1) \ar@{->}[ur] & & (-2,-3) \ar@{->}[ur] & & (2,1) \ar@{->}[ur]}
\]
\end{ex}

\section{Geometric model of $\C_m (A_n)$}

Let $n, m$ be positive integers and $Q$ be a Dynkin quiver of type $A_n$. We can assume that $Q$ has the following linear orientation $n \rightarrow n-1 \rightarrow \cdots \rightarrow 2 \rightarrow 1$. In this section, we give a geometric model for the orbit categories $\C_m (Q)=: \C_m(A_n)$. This model is based on the equivalence of categories given in theorem \ref{thm:doubleperp}. Indeed, given an arc $(t,u) \in \sfT$ and $n = \frac{u-t-1}{d}-1$, we can view the vertices in $]u,t[$ as vertices of a $(n+1)(|w|+1)-2$-gon. Each pair $\{u+1,u+2\}, \{u+2,u+3\}, \ldots, \{t-2,t-1\}$ and $\{t-1,u+1\}$ is viewed as an edge of the polygon. Each admissible arc of the $\infty$-gon lying in $\sfC_1$ can thus be seen as a diagonal of the polygon which divides it into two polygons whose number of vertices is divisible by $|d|$. Some of these diagonals can be edges of the polygon, when $w = -1$, in which case, a $2-$gon is considered a polygon.

We note that a different geometric model for $\C_1 (A_n)$ was given in \cite{CS2}. It is also interesting to note that the geometric model we will present here shares similarities with the geometric model for the higher cluster category of type $A$ introduced in \cite{BaurMarsh}. The results in this section can be easily proved using the same method as that in \cite{BaurMarsh}, and so the proofs will be omitted.

Let $\P_{n,m}$ be a regular $N$-gon, where $N = (n+1)(m+1)-2$, with vertices numbered clockwise from $1$ to $N$. All operations on vertices of $\P_{n,m}$ will be done modulo $N$. 

We define the stable translation quiver $\Gamma (n,m)$ as follows. The vertices of $\Gamma (n,m)$ are diagonals of $\P_{n,m}$ which divide $\P_{n,m}$ into two polygons whose number of vertices is divisible by $m+1$. We call these diagonals \textit{$(m+1)$-diagonals}. We denote a vertex of $\Gamma(n,m)$ by $\{i,j\}$ (or simply by $i,j$), where $i$ and $j$ are vertices of $\P_{n,m}$. The arrows of $\Gamma (n,m)$ are obtained in the following way: given two $(m+1)$-diagonals $D, D^\prime$ with a vertex $i$ in common, let $j$ and $j^\prime$ be the other vertices of $D$ and $D^\prime$ respectively. Then, there is an arrow from $D$ to $D^\prime$ in $\Gamma (n,m)$ if and only if $D^\prime$ can be obtained from $D$ by rotating clockwise $m+1$ steps around $i$.

Finally we define an automorphism $\tau: \Gamma (n,m) \rightarrow \Gamma (n,m)$ as follows: given an $(m+1)$-diagonal ${i,j}$, $\tau ({i,j}):= \{i-m-1, j-m-1\}$.

\begin{lemma}
The pair $(\Gamma (n,m), \tau)$ is a stable translation quiver.
\end{lemma}

\begin{prop}
The Auslander-Reiten quiver of $\C_m(A_n)$ is isomorphic to $\Gamma (n,m)$.
\end{prop}

\begin{cor}
The orbit category $\C_m(A_n)$ is equivalent to the additive hull of the mesh category of $\Gamma (n,m)$. 
\end{cor}

\begin{ex}
In the case when $n = 3$ and $m = 2$, $\P_{3,2}$ is a 10-gon and $\Gamma (3,2)$ looks as follows:
\[
\xymatrix@C=0.1cm{&                       & 1,9 \ar[dr]         &                       & 2,4 \ar[dr]       &                      & 5,7 \ar[dr]
&                      & 8,10 \ar[dr]       &                     & 1,3 \ar[dr] \\
                      & 1,6 \ar[dr]\ar[ur]  &                       & 4,9 \ar[dr]\ar[ur] &                      & 2,7 \ar[dr]\ar[ur] &
& 5,10 \ar[dr]\ar[ur] &                     & 3,8 \ar[dr]\ar[ur] &                          & 1,6 \ar[dr] \\
1,3 \ar[ur] &                       & 4,6 \ar[ur]  &                       & 7,9 \ar[ur] &                      & 2,10 \ar[ur]
&                      & 3,5 \ar[ur] &                     & 6,8 \ar[ur] &                     & 1,9}
\]
\end{ex}

\begin{thm}\label{combinatorial_presentation_hom_configs}
The $m-\Hom$-configurations in $\C_m(A_n)$ are in one-to-one correspondence with sets of $n$ noncrossing $(m+1)$-diagonals of an $N$-gon with no vertex in common. 
\end{thm}
\begin{proof}
Let $\sfC_1$ be as in theorem \ref{thm:doubleperp}, with $\mathtt{a} = (N+1,0)$ and $w = -m$. Clearly, the map $G: \Gamma(m,n) \rightarrow \sfC_1$, $\{i,j\} \mapsto (N+1-i,N+1-j)$, induces a bijection between $\C_m (A_n)$ and $\sfC_1$. One can check, doing a case-by-case analysis, that given $i \in \{0, \ldots, m\}$, we have $\Hom_{\C_m(A_n)} (\Sigma^i x, y) \ne 0$ if and only if $\Hom_{\sfT} (\Sigma^i (G(x)), G(y)) \ne 0$. Hence, it remains to check that a set $\sfH$ of noncrossing $(m+1)$-diagonals with no common vertex is maximal with respect to these properties if and only if it has cardinality $n$. This fact can be easily proved by induction on $n$.%%We will prove this by induction on $n$ (and fixed $m$). 
\end{proof}

\subsection{Classical noncrossing partitions and the case when $m = 1$:}\label{subsection:ncp}

In this subsection we restrict to the particular case where $m = 1$, and we compare this geometric model with the already existing one in \cite{CS2}. As a consequence, we will see the relationship between the two different combinatorial presentations of Hom-configurations in $\C_1 (A_n)$: the one given in theorem \ref{combinatorial_presentation_hom_configs} and the one given in \cite{CS2}.

First, let us recall the geometric model presented in \cite{CS2}. In that paper, the author considers oriented edges between vertices of a regular $n$-gon $\P_n$. Boundary edges and loops are included, and the edge oriented from $i$ to $j$ is denoted by $[i,j]$, or simply by $ij$.

Let $\Gamma^\prime (n) = \Gamma^\prime$ be the quiver defined as follows: the vertices are the set of all possible oriented edges between vertices of $\P_n$, including loops. The arrows are of the form $[i,j] \rightarrow [i+1,j]$, for $j \ne i+1$, and $[i,j] \rightarrow [i,j+1]$, for $i \ne j$, where $i+1, j+1$ are taken modulo $n$. This quiver is a stable translation quiver, with translation $\tau$ given by rotating edges in $\P_n$ through $2 \pi /n$ anticlockwise; in other words $\tau ([i,j]) = [i-1, j-1]$.

Without loss of generality, we assume that the bijection between $\Gamma^\prime$ and the AR-quiver of $\C_1(A_n)$ maps the simple module $S_1$ to the oriented arc $12$ (note that this defines the bijection completely). We can also assume, without loss of generality, that the bijection between $\C_1 (A_n)$ and $\Gamma (n,1)$ is such that the simple module $S_1$ is mapped to the arc $\{1,2\}$. 

\begin{remark}\label{thetwogeometricmodels}
It is easy to see that the map from $\Gamma^\prime_0$ to $\Gamma (n,1)_0$ defined by $ij \mapsto \{(2 (i-1 \mod n) +1) \mod 2n, (2 (j-1 \mod n) \mod 2n)\}$ induces an isomorphism of stable translation quivers. 
\end{remark}

\begin{ex}
Here is the isomorphism for when $n = 3$:
\[
\begin{tabular}{l c l}
\xymatrix@C=0.2cm{&                       & 11 \ar[dr]         &                       & 22 \ar[dr]       &                      & 33 \ar[dr]
&                      & 11\\
                      & 13 \ar[dr]\ar[ur]  &                       & 21 \ar[dr]\ar[ur] &                      & 32 \ar[dr]\ar[ur] &
& 13 \ar[ur] \\
12 \ar[ur] &                       & 23 \ar[ur]  &                       & 31 \ar[ur] &                      & 12 \ar[ur]} & \begin{tabular}{c} \\ \\ \\ \\ \\ $\mapsto$ \\ \end{tabular} & 
 \hspace*{-0.5cm} \xymatrix@C=0.1cm{&                       & 1,6 \ar[dr]         &                       & 2,3 \ar[dr]       &                      & 4,5 \ar[dr]
&                      & 1,6\\
                      & 1,4 \ar[dr]\ar[ur]  &                       & 3,6 \ar[dr]\ar[ur] &                      & 2,5 \ar[dr]\ar[ur] &
& 1,4 \ar[ur] \\
1,2 \ar[ur] &                       & 3,4 \ar[ur]  &                       & 5,6 \ar[ur] &                      & 1,2 \ar[ur]}
\end{tabular}
\]
\end{ex}

Hom-configurations in $\C_1(Q)$, where here $Q$ is any Dynkin quiver, are in bijection with positive noncrossing partitions \cite{CS1} (see also \cite{BRT} and \cite{Riedtmann}). In the case when $Q$ is of type $A_n$, positive noncrossing partitions are in one-to-one correspondence with classical noncrossing partitions of the set $\{1, \ldots, n\}$. These combinatorial objects, introduced by Kreweras \cite{Kreweras}, are defined to be partitions $\P = \{\B_1, \ldots, \B_m\}$ of the set $\{1, \ldots, n\}$ with the property that if $1 \leq a < b < c < d \leq n$, with $a, c \in \B_i$ and $b, d \in \B_j$, then $\B_i = \B_j$. We call $\B_i$ a \emph{block} of $\P$ for each $1 \leq i \leq m$; see figure \ref{Homconfigncp} for an example of the bijection between Hom-configurations in $\C_1(A_n)$ and noncrossing partitions when $n = 3$. 

\begin{figure}[!ht]
\includegraphics[scale=.5]{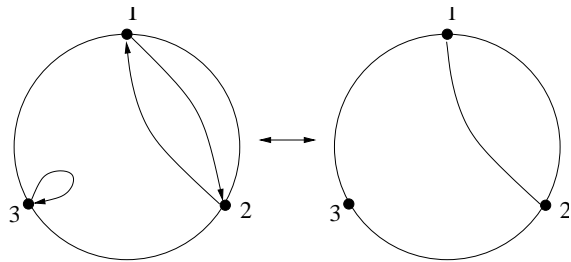}
\caption{The bijection between Hom-configurations and noncrossing partitions.}
\label{Homconfigncp}
\end{figure}

It is known that classical noncrossing partitions of $n$ are in one-to-one correspondence with noncrossing pair partitions (each block has cardinality two) of $2n$. The bijection is given by drawing around the outside of the boundaries of the convex hulls of the blocks of a noncrossing partition; see figure \ref{maprho} for an example. We denote this bijection by $\rho$. 

\begin{figure}[!ht]
\includegraphics[scale=.5]{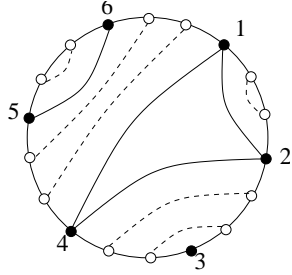}
\caption{The bijection between noncrossing partitions and noncrossing pair partitions.}
\label{maprho}
\end{figure}

Let us label the new vertices added to the $n-gon$ from $1^\prime$ to $2n^\prime$ clockwise around the circle, in such a way that $1^\prime$ is the first vertex after $1$. We note that the labelling is arbitrary, and it will depend only on the choice of the bijection between $\Gamma^\prime$ and $\Gamma (n,1)$. 

\begin{prop}
Let $\sfH$ be an Hom-configuration in $\C_1 (A_n)$. Let $\mathbb{H}_1$, respectively $\mathbb{H}_2$, be the set of vertices in $\Gamma^\prime$, respectively $\Gamma (n,1)$, corresponding to the indecomposable objects of $\sfH$. Then $\mathbb{H}_2 = \rho (\mathbb{H}_1)$. 
\end{prop}
\begin{proof}
This follows immediately from the definition of $\rho$ and the bijection between the two different geometric models given in remark \ref{thetwogeometricmodels}.
\end{proof}

\begin{figure}[!ht]
\includegraphics[scale=.5]{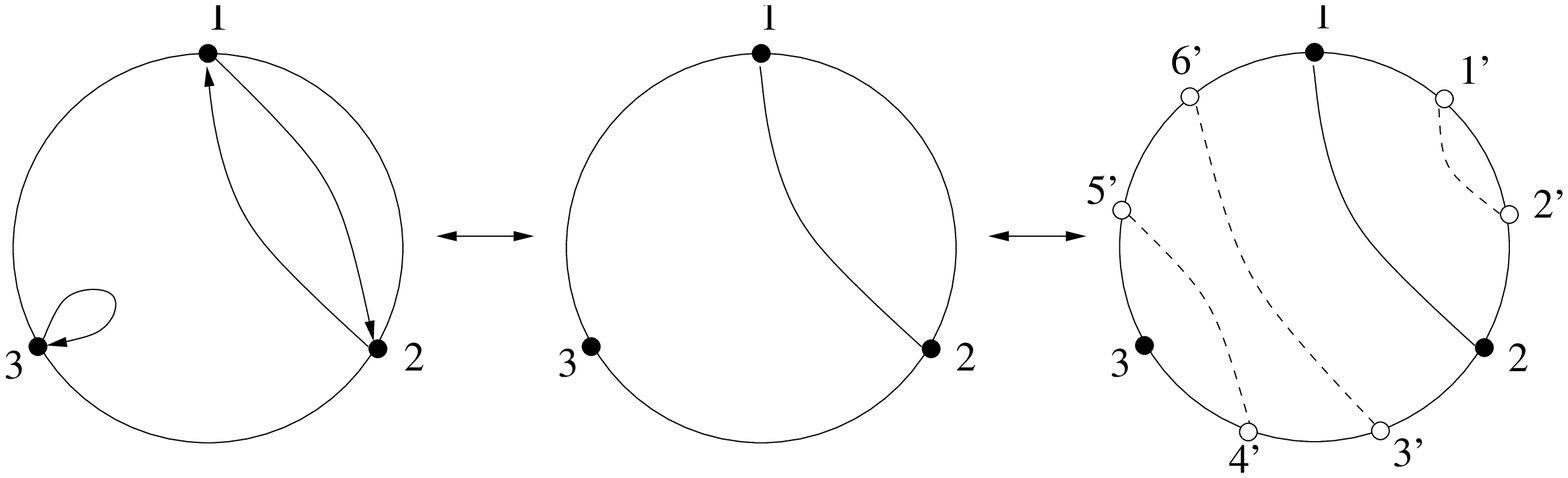}
\caption{The Hom-configuration $\{1, 23, \Sigma 2\}$ in both geometric models.}
\label{Homconfigdifmodels}
\end{figure}

\section{Noncrossing partitions of $\mathbb{Z}$ and $|-1|-\Hom$-configurations}

In this section, $\sfT$ will be $(-1)-$CY. We can easily extend the notion of classical noncrossing partition of the set $\{1, \ldots, n\}$ to ``classical noncrossing partitions of $\mathbb{Z}$'' as follows.

\begin{definition}
A \textit{noncrossing partition of $\mathbb{Z}$} is a set of pairwise disjoint non-empty subsets, called \textit{blocks}, of vertices of the $\infty$-gon, whose union is all of $\mathbb{Z}$, and such that no two blocks cross each other.
\end{definition}

We can relate $|-1|-\Hom$-configurations in $\sfT$ with noncrossing partitions of $\mathbb{Z}$, extending the bijection $\rho$ mentioned in subsection \ref{subsection:ncp}, as follows.

Let $\mathbb{Z}^\prime$ be another copy of $\mathbb{Z}$ in the $\infty$-gon, with vertices $n^\prime:=0.5 + 2n$, for each $n \in \mathbb{Z}$. Now, given an Hom-configuration $\sfH$ in $\sfT$, we define $f(\sfH)$ to be the partition $\P = \{B_1, B_2, \ldots\}$ of $\mathbb{Z}^\prime$, where the vertices in a block $B =\{a_1, a_2, \ldots\}$ are written in numerical order, i.e. $a_1 < a_2 < \ldots$, and are defined recursively as follows: let $a_i = n^\prime$, for some $n^\prime \in \mathbb{Z}^\prime$. If $n^\prime + 0.5$ is incident with an arc in $\sfH$ of the form $(m,n^\prime+0.5)$, then $a_{i+1} := m+0.5$. Otherwise, either $n^\prime + 0.5$ is incident with an arc in $\sfH$ of the form $(n^\prime+0.5,m)$ or $n^\prime+0.5$ is isolated. In these cases, we take $a_i$ to be the last element of its block. Informally speaking, this map is the \textit{infinite version} of the inverse of $\rho$. 

\begin{prop}
The map $f$ gives a one-to-one correspondence between $|-1|-\Hom$-configurations in $\sfT$ and noncrossing partitions of $\mathbb{Z}$ with at most one infinite block. 
\end{prop}
\begin{proof}
Let us check that this map is well defined, i.e. that the image of a Hom-configuration is indeed a noncrossing partition of $\mathbb{Z}$ with no more than one infinite block. 

Given a $|-1|-\Hom$-configuration, since $\sfH$ does not admit crossings nor arcs with a common vertex, it is clear that $f(\sfH)$ is indeed a noncrossing partition of $\mathbb{Z}$. Now, suppose for a contradiction, that there are two distinct blocks $\B_1$ and $\B_2$ which are infinite. Since $f(\sfH$) is noncrossing, these blocks must be of the form $\B_1 = \{a_1, a_2, \ldots\}$, $\B_2 = \{b_1, b_2, \ldots \}$, with $a_j > a_{j-1}, b_j < b_{j-1}$, for $j \geq 2$, and $b_1 < a_1$. Consider the vertices $x = b_1 + 0.5$ and $y = a_1-0.5$ of the $\infty-$gon. Note that, since $\B_1$ and $\B_2$ are infinite blocks, there cannot be an integer $z$ such that $z>y$ and $(z,y) \in \sfH$ or $z<x$ and $(y,z) \in \sfH$. Therefore, $y$ is either isolated or it is incident with an arc in $\sfH$ of the form $(y,z)$ with $x \leq z \leq y-1$. 

Suppose the latter holds. Then, by definition of $f$, $z-0.5$ must lie in the same block as $a_1$. But $z-0.5 < a_1$, contradicting the assumption that $a_1$ is the lower bound of $\B_1$. Therefore, $y$ must be isolated. Using a similar reasoning for the vertex $x$, and the fact that there cannot be more than one isolated vertex, we conclude that $x$ must be incident with an arc in $\sfH$ of the form $(z,x)$ with $x+1 \leq z < y$. But again, by definition of the map $f$, $z + 0.5 \in \B_2$, contradicting the assumption that $b_1$ is the upper bound of this block. Therefore, $f(\sfH)$ must have at most one infinite block. 

It is clear that $f$ is indeed a bijection.  
\end{proof}

It is natural to ask what \textit{special} property of noncrossing partitions corresponds to $|-1|-$Riedtmann configurations.

\begin{prop}
$|-1|-$Riedtmann configurations in $\sfT$ are in one-to-one correspondence with noncrossing partitions of $\mathbb{Z}$ whose infinite block, if it exists, does not have upper or lower bound.
\end{prop}
\begin{proof}
We argue by contradiction. Let $\sfH$ be a $|-1|$-Riedtmann configuration and suppose, without loss of generality, that $f(\sfH)$ has an infinite block $\B$ with lower bound $a$. Since $\B$ is infinite, there is no $x > a$ such that $(x, a-0.5) \in \sfH$. On the other hand, if there is an arc in $\sfH$ of the form $(a-0.5, x)$, then $x < a$ and $x-0.5 \in \B$, a contradiction. Therefore, $a-0.5$ must be an isolated vertex. However, this cannot happen as $\sfH$ is a $|-1|$-Riedtmann configuration.   
\end{proof}

Note that, when we defined $f$, we could have chosen another copy of $\mathbb{Z}$, namely $\mathbb{Z}^{\prime\prime} = \{n^{\prime\prime} := 2n-0.5 \mid \, n \in \mathbb{Z}\}$. Given a noncrossing partition $\P$ of $\mathbb{Z}^\prime$, its \textit{Kreweras complement} $K(\P)$ is the unique maximal partition of $\mathbb{Z}^{\prime\prime}$ such that $\P \cup K(\P)$ is a noncrossing partition of $\mathbb{Z}^\prime \cup \mathbb{Z}^{\prime\prime}$. 

\begin{remark}
Choosing the other possible copy of $\mathbb{Z}$ just means we get the Kreweras complement of the original noncrossing partition. In other words, if $g$ is the map from the set of Hom-configurations to the set of noncrossing partitions of $\mathbb{Z}^{\prime\prime}$ defined in the same way as $f$, then we have $g = K f$.
\end{remark}

\begin{ex}
Consider again the canonical examples of Hom-configurations mentioned in example \ref{canonicalexamples}. We have $f (\sfH_1) = \{ \{\mathbb{Z}\} \}$ and $g (\sfH_2) = \{ \{n\}_{n \in \mathbb{Z}}\}$. Assume, for simplicity, that $i = 0$ in $\sfH_2$. Then $f(\sfH_2) = \{ \{\mathbb{N}_0\}, \{n\}_{n \in \mathbb{Z}_{<0}}\}$, and $g(\sfH_2) = \{ \{\mathbb{Z}_{\leq 0} \}, \{n\}_{n \in \mathbb{N}} \}$.
\end{ex}

\subsection*{Acknowledgments.}
This work was partly carried out while the author was a Riemann fellow in the Institut f\"ur Algebra, Zahlentheorie und Diskrete Mathematik (IAZD), at the Leibniz Universit\"at Hannover. The author would like to express her gratitude to the Riemann Centre and to IAZD for this opportunity and for their kind hospitality. The author would also like to thank Peter J\o rgensen for his suggestion to work with these triangulated categories, David Pauksztello for reading a preliminary version of this paper and Robert Marsh for useful discussions. She would also like to thank Funda\c{c}\~ao para a Ci\^encia e Tecnologia, for their financial support through Grant SFRH/BPD/90538/2012.

\end{document}